\documentclass{amsart}

\usepackage{amssymb}
\usepackage{bm}
\usepackage{graphicx}
\usepackage[centertags]{amsmath}
\usepackage{amsfonts}
\usepackage{amsthm}

\newtheorem{theorem}{Theorem}

\newtheorem{coroll}{Corollary}

\newtheorem{claim}{Claim}
\newtheorem{lemma}{Lemma}
\newtheorem{proposition}{Proposition}
\newtheorem{definition}{Definition}
\theoremstyle{definition}

\begin{document}
 
\title[An abstract approach in canonizing topological Ramsey spaces]{An abstract approach in canonizing topological Ramsey spaces}

\author{ Dimitris VLITAS}

\address{Department of Mathematics, University of Toronto, 40 St. George Street, Toronto, Ontario, Canada M5S 2E4}
\email{vlitas@math.univ-paris-diderot.fr}

 %------------------------Abstract-------------------------------%
\begin{abstract} In \cite{To} S. Todorcevic introduced the notion of a topological Ramsey space and a list of axioms required to be satisfied by any such a space. Here we show that any topological Ramsey space that satisfies a strengthened version of  one of the required axioms and a very natural assumption, admits canonization theorem. 
\end{abstract}
\maketitle

\section{Introduction}

In \cite{To} S. Todorcevic introduced the notion of a topological Ramsey space. Topological Ramsey spaces are structures of the form $\langle \mathcal{R},\leq,r, \rangle$, satisfying certain conditions (see in the next section). Someone should think of $\mathcal{R}$ as family of infinite sequence of objects and the function $r$ as finite approximations of these infinite sequences. Then $\leq$ is a quasi-ordering on $\mathcal{R}$. These spaces admit a natural topology and they are required to satisfy four axioms, $A.1-A.4$. As a consequence of these axioms one gets that Ramsey subsets of $\mathcal{R}$ are exactly those with the Baire property and meager sets are Ramsey null. Many of the well known spaces can be seen as instances of topological Ramsey spaces. The most well known example is the Ellentuck space $\langle \mathbb{N}^{[\infty]},\subseteq, r\rangle$. For $A\in \mathbb{N}^{[\infty]}$, $r_n(A)$ is the initial segment of $A$ formed by taking the first $n$-elements of $A$. 

Canonical results in Ramsey theory try to  describe equivalence relations in a given Ramsey structure,
based on the underlying pigeonhole principles. The first example of them is the classical Canonization
Theorem  by P. Erd\H{o}s and R. Rado \cite{Er-Ra} which can be presented as follows: Given $\alpha\le \beta
\le \omega$ let
$$\binom{\beta}{\alpha}:=\{f (\alpha)  \,:\, f:\alpha \rightarrow \beta \text{ is strictly increasing} \}.$$
The previous is commonly denoted by $[\beta]^{\alpha}$.   Then for any $n<\omega$  and any finite coloring of
$\binom{\omega}{n}$ there is an isomorphic copy $M$ of $\omega$ (i.e. the image of a strictly increasing
$f:\omega\rightarrow\omega$) and some
   $I\subseteq n(:=\{0,1,\dots,n-1\})$ such that any two  $n$-element subsets have the same color  if and only if they agree
  on the corresponding relative positions given by $I$.

This   was extended by P. Pudl\'ak and V. R\"odl in \cite{Pu-Ro}  for colorings of  a given \emph{uniform}
family $\mathcal{G}$ of finite subsets of $\omega$ by showing that given any
coloring of  $\mathcal{G}$, there exists $A$ an infinite subset of $\omega$, a
uniform family $\mathcal{T}$ and a mapping $f:\mathcal{G}\to \mathcal{T}$ such that $f(X)\subseteq X$ for all
$X\in \mathcal{G}$ and such that any two $X,Y\in \mathcal{G}\upharpoonright A$ have the same color   if and
only if $f(X)=f(Y)$.

The P. Erd\H{o}s-Rado result deals with equivalence relations on the family of $k$ approximations of elements of members of $\mathbb{N}^{[\infty]}$. The Pudl\'ak-R\"odl result deals with equivalence relations on uniform families of finite approximations of elements of members of $\mathbb{N}^{[\infty]}$. In this paper we are going to generalize the above results to any topological Ramsey space. Namely that given a family of finite approximations $\mathcal{F}$ of $\mathcal{R}$, ( see Definition $1$) and an equivalence relation $f:\mathcal{F}\to \omega$ on it, there exists an $X\in  \mathcal{R}$ and a map $\phi$, ( see Definition $2$ ) so that for any $s,t\in \mathcal{A}X$, $s,t\in \mathcal{F}$ it holds that $f(s)=f(t)$ if and only if $\phi(s)=\phi(t)$.

\section{Background material}

  Topological Ramsey spaces are spaces of the form $\langle \mathcal{R}, \leq, r \rangle$, where $r:\mathcal{R}\times \omega \to \mathcal{AR}$ is a map that gives us the sequence $r(\cdot, n)=r_n(\cdot)$ of approximation mappings. The basic open sets are of the form: $$[s,X]=\{ Y\leq X: ( \exists n) r_n(Y)=s\},$$ for $s\in \mathcal{AR}$ and $X\in \mathcal{R}$. If $r_n(X)=s$ we write $s\sqsubseteq X$. The axioms required to be satisfied by any such a  space in order to be topological Ramsey space are the following.

   $\boldsymbol{A.1.}$
    
     Let $X,Y \in \mathcal{R}$.
     
    \begin{enumerate}
        \item{} $r_0(X)=\emptyset$ for all $X\in \mathcal{R}$.\\
        \item{} $X\neq Y$ implies $r_n(X)\neq r_n(Y)$ for some $n\in \omega$.\\
       \item{}  $r_n(X)=r_m(Y)$ implies $n=m$ and $r_k(X)=r_k(Y)$ for all $k<n$.
   \end{enumerate}

  $\boldsymbol{A.2.}$

  There is a quasi-ordering $\leq_{fin}$ on $\mathcal{AR}$ such that
  
   \begin{enumerate}
  
    \item{} For any $s\in \mathcal{AR}$ the set\\ $\{\, t \in \mathcal{AR}: t\leq_{fin} s\, \}$ is finite.\\
  \item{} For any $X,Y\in \mathcal{R}$, $X\leq Y$ if and only if \\ $(\forall n)(\exists m) r_n(X)\leq_{fin}r_m(Y)$.\\
  \item{}  For all $s,t\in \mathcal{AR}$\\
  $[s\sqsubseteq t\wedge t\leq_{fin}t' \to \exists \tilde{t} \sqsubseteq t':\, s\leq_{fin} \tilde{t}]$.\\
  \end{enumerate}
  
  $\boldsymbol{A.3}$
  
  Let $s\in \mathcal{AR}$, $X,Y,Z\in \mathcal{R}$.
  \begin{enumerate}
  \item{} If $[s,X]\neq \emptyset $ then $[s,Y]\neq \emptyset$ for all $Y\in [s,X]$.\\
  \item{} $X\leq Y$ and $[s,X]\neq \emptyset$ imply that there is $Z \in [s,Y]$ such that $\emptyset \neq [s,Z]\subseteq [s,X]$.\\
  \end{enumerate}

 $ \boldsymbol{A.4}$\\
 Let $X \in \mathcal{R}$, $s\in (\mathcal{AR})_n$, $[s,X]\neq \emptyset$ and  $\mathcal{O}\subseteq (\mathcal{AR})_{n+1}$. There exists $Y\in [s,X]$ such that: 
 $$r_{n+1}[s,Y]\subseteq \mathcal{O} \text{ or } r_{n+1}[s,Y]\subseteq \mathcal{O}^{c},$$ 
 where  $r_{n+1}[s,Y]=\{\, t \in (\mathcal{AR})_{n+1} : s\sqsubseteq t \}$. \\

  We say that a subset $\mathcal{H}$ of $\mathcal{R}$ is \emph{Ramsey} if for every $[s,X]\neq \emptyset$ there is a $Y\in [s,X]$ such that  either $[s,Y]\subset \mathcal{H}$ or $[s,Y]\subset \mathcal{H}^c$, and    $\mathcal{H}$  is \emph{Ramsey null} if for every $[s,X]\neq \emptyset$, there is $Y$ such that $[s,Y]\cap \mathcal{H}=\emptyset$. In \cite{To} it is shown  that if $\langle \mathcal{R},\leq,r\rangle$ is a topological Ramsey space, then the Ramsey subsets of $\mathcal{R}$ are exactly those with the Baire property. Moreover meager sets are Ramsey null. Then one gets as an immediate consequence the following two corollaries.
  
  \begin{coroll}
  Let $X\in \mathcal{R}$, $n<\omega$ and $c:\mathcal{A}X_n\to l$ be a finite coloring. There exists an $Y\leq X$ so that $c\upharpoonright \mathcal{A}Y_n$ is constant. 
  \end{coroll}
   and also
   
   \begin{coroll}
   Given $c:[s,X]\to l$ a finite Suslin measurable coloring, there exists $Y\in [s,Y]$ so that $c\upharpoonright [s,Y]$ is constant.
   \end{coroll}
   
    Recall that a map $f:X \to Y$ between two topological spaces is Suslin measurable, if the preimage $f^{-1}(U)$ of every open subset $U$ of $Y$ belong to the minimal $\sigma-$field of subsets of $X$ that contains its closed sets and it is closed under the Suslin operation \cite{Ke}.

For $s\in \mathcal{AR}$ and $X\in \mathcal{R}$ we define the depth of s in X as follows:
\begin{equation*}
depth_X(s)=
\begin{cases}
\min\{k: s \leq_{fin} r_k(X) \} & \text{ if } (\exists k) s \leq_{fin} r_k(X),\\
\infty & \text{otherwise.}
\end{cases}
\end{equation*}

From now on we will be working with topological Ramsey spaces that admit a maximal element $U$. Therefore we are looking on the structures satisfying the above axioms of the form $$\langle U, \leq, r \rangle$$
where the elements of the space, are the $\it{ reducts }$ of $U$, $X\leq U$. Let $$\mathcal{A}U_n=\{ r_n(X): X\leq U\}$$ for $n\in \omega$ and $$\mathcal{A}U=\cup_{n\in \omega} \mathcal{A}U_n.$$ Similarly for any $X\leq U$, $n\in \omega$, we define $\mathcal{A}X_n=\{ r_n(Y): Y\leq X\}$ and $\mathcal{A}X=\cup_{n\in \omega} \mathcal{A}X_n$.
For $s\in \mathcal{A}U$ by $|s|$ we denote its length, i.e. the unique number $n$ so that $s=r_n(X)$, for an $X\leq U$. For $X\leq U$ by $X(n)$ we denote the sequence of objects $r_n(X)\setminus r_{n-1}(X)$, for $n\geq 1$. Therefore for each $X\leq U$ we get a countable sequence $(X(n))$. Similarly with each $t\in \mathcal{AU}_n$ we associate the sequence $t(i)_{i\in n}$ of length $n$. 
Let $U[k,l)=\cup_{n\in [k,l)} U(n)$. 
Observe that axiom $A.2$ implies that if $X\leq U$, $s=r_{n-1}(X)$ and $s\leq_{fin} r_k(U)$, then $ t=r_{n}(X)$ is of the form $t\leq_{fin} r_l(U)$ for some $l> k$. As a consequence to get $X(n)$ we have used the levels $U(k), \dots, U(l-1)$ of $U$. 
For $s\in \mathcal{A}U$, with $depth_U(s)=k$, we define $$X[s]=\{ t\in \mathcal{A}U: s\sqsubseteq t \text{ and } t\leq_{fin} r_j(Y), \text{for some } Y\leq X, j< \omega \}.$$ Observe that $X[s]\neq \emptyset$ if and only if $[s,X]\neq \emptyset$. In this case we say that $X$ is {\it{ compatible }} with $s$. Notice also that for every $t\in U[s]$ there exists a set $\mathcal{X}\subseteq [k,l)$, where $depth_U(t)=l$, so that $t\setminus s$ is made out of $\cup_{n\in \mathcal{X}}U(n)$.

Next given $s$ and $X$, so that $[s,X]\neq \emptyset$, by $X/s$ we denote $X\setminus s$.

\section{Main theorem}

We introduce the notion of a $\it{Front}$.

\begin{definition} A family $\mathcal{F}$ of finite approximations of reducts of $U$ is called a front, if for every $X\leq U$, there exists $s\in \mathcal{F}$ so that $s\sqsubseteq X$ and for any two distinct $s,t\in \mathcal{F}$, is not the case that $s\sqsubseteq t$.
\end{definition}

We distinguish a specific case of fonts the families of finite approximations of length $n$. Namely $\mathcal{A}U_n=\{ s: s=r_n(X), X\leq U \}$.

 Given a front $\mathcal{F}$ on $[\emptyset, X]$, $X\leq U$ we introduce $\hat{\mathcal{F}}$ defined as follows: $$\hat{\mathcal{F}}=\{ t\in \mathcal{A}U: \exists s\in \mathcal{F}, t\sqsubseteq s\}$$ observe that $\emptyset \in \hat{\mathcal{F}}$. For $t\in \hat{\mathcal{F}}\setminus \mathcal{F}$ $$\mathcal{F}_t=\{s\in \mathcal{F}: t\sqsubseteq s \}.$$
 
 Notice that $\mathcal{F}_t$ is a front on $U/t$.
 For $Y\leq X$ $$\mathcal{F}\upharpoonright  Y  =\{t\in \mathcal{F}: t\in \mathcal{A}Y', Y'\leq Y\},$$
 
 $$\hat{\mathcal{F}}\upharpoonright Y  =\{t\in \hat{\mathcal{F}}: t\in \mathcal{A}Y', Y'\leq Y\}.$$

Next we introduce a stronger version of $ A.4$ as follows:\\
 $ \boldsymbol{A.4}^\star$\\
 Let $s\in \mathcal{A}U_n$, $depth_U(s)=k$, $X\leq U$ with $[s,X]\neq \emptyset$ and a coloring $$c:[s,X]_{n+1}\to \omega,$$ where $[s,X]_{n+1}=\{t\in \mathcal{A}X_{n+1}: s\sqsubseteq t\}=r_{n+1}[s,X]$ be given. There exists a map $$\phi_s:[s,U]_{n+1}\to \mathcal{P}( (\cup [s,U]_{n+1})\bigcup [k,\infty))$$ so that $\phi_s(p)\subseteq p(n)\cup [k,l)$, where $l=depth_U(p)$, and $Y\in [s,X]$ so that for all $p,q \in [s,Y]_{n+1}$ it holds that $c(p)=c(q)$ if and only if $\phi_s(p)=\phi_s(q)$. We will call such a mapping $\phi_s$ inner for $s$. 
  
 In other words, there exists a reduct $Y$, where the coloring $c$ dependents only on a subset of $p(n)$, for any $p\in [s,Y]_{n+1}$, and the levels $U[k,l)$ needed to get $p(n)$. In some sense $\phi_s$ gives us a subset of the information coded by $p(n)$. 
 
 $A.4^\star$ essentially deals with the length one extensions of initial segments. This necessitates to introduce the set of all length one extensions of all members of $\mathcal{A}X$. Let $$\mathcal{L}X_1=\{ w: \exists s\in \mathcal{A}X, s\cup w\in [s,X]_{|s|+1}  \}.$$ Observe that $\mathcal{A}X_1\subseteq \mathcal{L}X_1$, cause $r_0(X)=\emptyset$. Let $$ \mathcal{L}X_n=\{ w:\exists t\in \mathcal{A}X, t\cup w\in [t,X]_{|t|+n} \}$$ and $$\mathcal{L}X=\cup_{n\in \omega}\mathcal{L}X_n.$$ 
 Next we introduce a partial ordering on $\mathcal{L}X$ as follows. 
 Given $w,v\in \mathcal{L}X$ we write $w\leq v$ if there exists $s\in \mathcal{A}X$ so that $s\cup w\cup v\in \mathcal{A}X$. Let $w_0,\dots, w_n$ elements of $\mathcal{L}X$ so that either $w_0\leq \dots \leq w_n$ and $\exists t\in \mathcal{A}X$ so that $t\cup w_0\cup \dots \cup w_n\in \mathcal{A}X$ or $w_i\in X(n)$, for all $i \leq n$, and $\exists t\in \mathcal{A}X$ such that  $t\cup (w_0, \dots, w_n)\in \mathcal{A}X_{|t|+1}$. Then $\langle w_0,\dots w_n\rangle_s \in \mathcal{L}X_1$ is defined to be the set of all end extension of $s$, $s\leq_{fin} t$, made out of $w_0,\dots , w_n$. When we say that the end extension in made out of $w_0,\dots , w_n$, we mean all $w_i, i \leq n$, are needed, not a proper subset is sufficient. Notice that if $s\cup w\in [s,X]_{|s|+1}$, then $\langle w\rangle_s=w$ in this case we say that $\langle \rangle_s$ acts {\it{trivially }} on $w$.

 To state the main theorem of this paper we need the following definition.
 
 \begin{definition}Let $\mathcal{F}$ be a front on $[\emptyset, U]$ and let $\phi$ be a function on $\mathcal{F}$. We call $\phi$ is Inner if for every $t=t(i)_{i\in n}\in \mathcal{F}$ we have that $$\phi(t)= (\phi_{s_0}(w_0), \dots ,\phi_{s_m} (w_m) ).$$

  Where $\phi_{s_0},\dots, \phi_{s_m}$ are inner maps, for $s_0=r_{h_0}(t),\dots, s_m=r_{h_m}(t)$ respectively, where $h_0<\dots <h_m<n$. Also  $w_0\in \langle t(i_0), \dots ,  t(i_l)\rangle_{t(h_0)}, \dots, w_m\in \langle t(j_0), \dots , t(j_m)\rangle_{t(h_m)}$, $\{i_0, \dots, i_{l_0}, \dots, j_0,\dots,j_{l_d},  h_0,\dots, {h_m}\}\subseteq  n$, and every $t_i$, $i<n$ appears in at most one combination in $\{\langle t(i_0), \dots ,  t(i_l)\rangle_{t(h_0)}, \dots, \langle t(j_0), \dots , \\ t(j_m)\rangle_{t(h_m)} \}$.
\end{definition}

  Now we can state the main theorem of this paper.

\begin{theorem} Given a topological Ramsey space $\langle U, \leq, r\rangle$, that satisfies $\boldsymbol{A.4}^\star$, a front $\mathcal{F}$ on $[\emptyset, U]$ and a coloring $f:\mathcal{F} \to \omega$, there exists $X\leq U$ and an Inner map $\phi$ on $\mathcal{F}\upharpoonright X$, so that for every $s,t \in \mathcal{F}\upharpoonright X$ it holds that $f(s)=f(t)$ if and only if $\phi(s)=\phi(t)$.
\end{theorem}

First we prove the following proposition.

\begin{proposition} Suppose the topological Ramsey space $\langle U, \leq, r\rangle$ has the property that given a property $\mathcal{P}(\cdot, \cdot)$, $s\in \mathcal{A}U$ and $X\leq U$, there exists $Y\leq X$ so that $\mathcal{P}(s,Y)$. Then there exists $Z\leq U$ such that for any $s\in \mathcal{A}Z$ it holds that $\mathcal{P}(s,Z)$.

Similarly for properties of the form $\mathcal{P}(\cdot,\cdot,\cdot)$. If given $s,t\in \mathcal{A}U$ and $X\leq U$, there exists $Y\leq X$ so that $\mathcal{P}(s,t, Y)$. Then there exists $Z\leq U$ so that $\mathcal{P}(s,t,Z)$ for all $s,t\in \mathcal{A}Z$.
\end{proposition}

\begin{proof} 

Let $t_0=r_0(U)$ and $U$. There exists $X_0\leq U$ so that $\mathcal{P}(t_0,X_0)$. Set $t_1=r_1(X_0)$. Consider the finite set $A_0=\{ z_i\in \mathcal{A}U: z_i\leq_{fin} t_1, i< l\}$. For every $z_i\in A_0$ there exists $Y\leq X_0$ so that $\mathcal{P}(z_i,Y)$. After considering all $z_i, i\in l$, we get $X_1\leq X_0$ so that for every $z_i\in A_0$, $\mathcal{P}(z_i,X_1)$ holds. Set $t_2=r_2(X_1)$. Suppose we have constructed $t_n$ and $X_n$. Set $t_{n+1}=r_{n+1}(X_n)$. Consider $A_{n}=\{z_i\in \mathcal{A}U: z_i\leq_{fin} t_{n+1}, i<l'\}$. For every $z_i\in A_{n}$ there exists $Y\leq X_n$ so that $\mathcal{P}(z_i,Y)$. Therefore we get $X_{n+1}\leq X_n$ so that for any $z_i\in A_{n}$ we have $\mathcal{P}(z_i,X_{n+1})$. Set $t_{n+2}=r_{n+2}(X_{n+1})$. Proceed in that manner. Observe that for all $n\in \omega$ $t_n\sqsubset t_{n+1}$. Set $Z=\cup_{n\in \omega}t_n$. \\

 Now we prove similarly the second statement of our proposition. Let $t_0=r_0(U)$ and $t_1=r_1(U)$ and $U$. There exists $X_1\leq U$ so that $\mathcal{P}(t_0,t_1,X_1)$. Let $t_2=r_2(X_1)$. Consider the finite set $A_2=\{ z\in \mathcal{A}U: z\leq_{fin} t_2\}$. For any $(s,t)\in [A_2]^2$, there exists $Y\leq X_1$ so that $\mathcal{P}(s,t,Y)$. By exhausting all possible such a pairs we get $X_2\leq X_1$ such that for any $(s,t)\in [A_2]^2$ it holds that $\mathcal{P}(s,t,X_2)$. Set $t_3=r_3(X_2)$. Suppose we have constructed $t_n$ and $X_n$. Let $t_{n+1}=r_{n+1}(X_n)$ and $A_{n+1}=\{z\in \mathcal{A}U: z \leq_{fin} t_{n+1}\}$. For any pair $(s,t)\in [A_n]^2$ there exists $Y\leq X_n$ so that $\mathcal{P}(s,t,Y)$ holds. After considering all possible such a pairs, we get $X_{n+1}$ such that for any $(s,t)\in [A_{n+1}]^2$ it holds that $\mathcal{P}(s,t,X_{n+1})$. Set $t_{n+2}=r_{n+2}(X_{n+1})$. Observe that for every $n\in \omega$, $t_{n}\sqsubset t_{n+1}$. Let $Z=\cup_{n\in \omega} t_n$. 
\end{proof}

 Next we make the following definition.

\begin{definition}
Let $f$ be an equivalence relation on a front $\mathcal{F}$ on $[\emptyset, U]$. Given $X\leq U$ and $s,t\in\hat{ \mathcal{F}}\setminus \mathcal{F}$ we say that $X$ separates $s$ with $t$ if for every $Y\leq X$ and $s',t'\in \mathcal{F}\upharpoonright Y$ where $s\sqsubseteq s'$, $t\sqsubseteq t'$ it holds that $f(s')\neq f(t')$. If there is no $Z\leq X$, that separates $s$ with $t$ we say that $X$ mixes $s$ with $t$. We say that $X$ decides for $s$ and $t$, if $X$ either mixes or separates them.
\end{definition}

Observe that $X$ mixes $s$ with $t$, if for every $Y\leq X$, there exists $s',t'\in \mathcal{F}\upharpoonright Y$, $s\sqsubseteq s', t\sqsubseteq t'$, so that $f(s')=f(t')$.
The following proposition follows directly from the definitions.

\begin{proposition} The following hold.
\begin{enumerate}
\item{} If $X$ mixes (separates) $s$ with $t$, so does any reduct $Y\leq X$.
\item{} For every $s,t\in \hat{\mathcal{F}}\setminus \mathcal{F}$ if for any $w\in [s,X]_{|s|+1}$ there exists $v\in [t,X]_{|t|+1}$ so that $X$ mixes $s\cup w$ with $t\cup v$, then $X$ also mixes $s$ with $t$.
\end{enumerate}

Next we observe the following.

\end{proposition}

\begin{lemma}(Transitivity of mixing) Let $s,t,w\in \hat{\mathcal{F}}\setminus \mathcal{F}$, where the following holds $depth_X(s)=depth_X(t)=depth_X(t')<\omega$. If $X$ mixes $s$ with $t$ and $t$ with $t'$, it also mixes $s$ with $t'$.
\end{lemma}

\begin{proof} Suppose that $X$ mixes $s$ with $t$ and $t$ with $t'$, but $X$ separates $s$ with $t'$. Consider the two-coloring $c_1:[t', X]_{n+1} \to 2$ defined by 
\begin{equation*}
c_1(p)=
\begin{cases}
1 & \text{ if } \exists q\in [t,X]_{n+1}, \text{ and } $X$ \text{ mixes } p \text { with } q,\\
0 & \text{ otherwise.}
\end{cases}
\end{equation*}
The fact that $\langle U, \leq, r\rangle$ is a topological Ramsey space, gives us a $Y\in [t', X]$ so that $c_1\upharpoonright [t',Y]_{n+1} =1$. Similarly we consider the two-coloring $c_2:[t,Y]_{n+1}\to 2$ defined by:
\begin{equation*}
c_2(p)=
\begin{cases}
1 & \text{ if } \exists q\in [s,Y]_{n+1}, \text{ and } $Y$ \text{ mixes } p \text { with } q,\\
0 & \text{ otherwise.}
\end{cases}
\end{equation*}

 which gives us a $Z\in [t, Y]$ so that $c_2\upharpoonright [t,Z]_{n+1}=1$. But this implies that $Z\leq X$ mixes $s$ with $t'$, a contradiction.

\end{proof}

The condition $depth_X(s)=depth_X(t)=depth_X(w)<\omega$ is necessary for the transitivity of mixing to be valid. In \cite{Vlit} where the topological Ramsey space of $\langle FIN_k^{[\infty]},\leq,r \rangle$ is examined, we see that there are colorings where the corresponding notion of mixing is not transitive. An instance of such a coloring is as follows. Let $FIN$ be the set of all non empty finite subsets of $\omega$. An element $X$ of $FIN^{[\infty]}$ is a sequence $X=(x_n)_{n\in \omega}$ so that $x_n\in FIN$, $\max  x_n< \min  x_{n+1}$, for all $n\in \omega$. We write $x_n<x_{n+1}$ to show that $\max  x_n< \min  x_{n+1}$. Let $\langle X\rangle=\{ x_{n_0}\cup \dots \cup x_{n_k} : n_0<\dots n_k \}$. For $X,Y\in FIN^{[\infty]}$, set $Y\leq X$ if $y_n\in \langle X\rangle$ for all $n\in \omega$. Finally we define that $r_n(X)=(x_i)_{i<n}$ and $r=\cup_{n\in \omega} r_n$. Then $\langle X, \leq, r \rangle$ becomes a topological Ramsey space, see \cite{To} for a full exposition.
Let $f: \mathcal{A}X_2\to \omega$ defined by $c(x_0,x_1)=x_0\cup x_1$. Consider $s=x_0$, $t=x_0\cup x_2$ and $t'=x_0\cup x_1\cup x_2$. Notice that $X$ mixes $s$ with $t$ and $s$ with $t'$, but $X$ does not mixes $t$ with $t'$. 

\begin{proposition}There exists $X\leq U$ that decides for all $s,t\in\hat{ \mathcal{F}}\upharpoonright X$.
\end{proposition}

\begin{proof} Given $s,t$ and $Y\leq U$ it suffices to show that there exists $Z\leq Y$ which decides for $s$ and $t$. Then the statement of the this proposition will follow from Proposition $1$ and the property $\mathcal{P}(s,t,Z)$ stating that $Y$ decides for $s$ and $t$. Consider the two-coloring: $c':[\emptyset, Y]\to 2$ defined by 
\begin{equation*}
c'(Y')=
\begin{cases}
1 & \text{ if } \exists p\in [s,Y']_{|s|+1}, q\in [t,Y']_{|t|+1} \text { so that }  Y' \text{ mixes } p \text{ with } q,\\
0 & \text{ otherwise.}
\end{cases}
\end{equation*}

The fact that $\langle Y, \leq, r\rangle$ is a topological Ramsey space, provides us with $Z$ that either mixes $s$ with $t$, in the case that $c'\upharpoonright [\emptyset, Z]=1$, or separates them, in the case that $c'\upharpoonright [\emptyset, Z]=0$.

\end{proof}

Now we prove Theorem $1$.
\begin{proof} Let $f:\mathcal{F}\to \omega$ be given as in Theorem $1$. Assume first that the notion of mixing introduced in Definition $3$ is transitive. In the next subsection we deal with the non-transitive case. 

Observe that axiom $A.4^\star$ for any $s\in \hat{ \mathcal{F}}\setminus \mathcal{F}$, $|s|=n$ and $X$ with $[s,X]\neq \emptyset$, provides us with a $Y\in [s, X]$ and  $\phi_s$ so that for $p,q \in [s,Y]_{n+1}$, $Y$ mixes $ p$ with $ q$ if and only if $\phi_s(p)=\phi_t(q)$. This is done by considering the coloring $c:[s,U]_{n+1}\to \omega$ defined by $c(p)=c(p')$ if and only if $U$ mixes $p$ with $p'$. 
By Propostiton $1$ there exists $X\leq U$ so that for every $s\in \hat{ \mathcal{F}}\setminus \mathcal{F}\upharpoonright X$ there exists such a $\phi_s$. We assume that $U$ has this property, instead of one of its reducts. Therefore we assume that $U$ decides any $s,t\in \mathcal{A}U$ and for any $s\in \hat{ \mathcal{F}}\setminus \mathcal{F}$, $|s|=n$, $\phi_s$ defines an equivalence relation on $[s,U]_{n+1}$.

Notice that for $s\in  \hat{ \mathcal{F}}\setminus \mathcal{F}$, if $\phi_s=\emptyset$, then $U$ mixes $s\cup w$ with $s\cup v$ for every $w,v\in [s,U]_{1}$.

Suppose now that our topological Ramsey space $\langle U, \leq, r\rangle$ has the property that given any $s,t\in \mathcal{A}U_n$, $Y\leq U/(s,t)$, then $Y_s=s\cup (Y[s])$ and $Y_t=t\cup (Y[t])$, where $Y[s]\neq \emptyset$, $Y[t]\neq \emptyset$. In other words $Y$ is compatible with both $s$ and $t$.
Instances of such a topological Ramsey spaces are the $\langle \mathcal{R}_1, \leq, r\rangle$  \cite{Do-To} and $\langle FIN_k^{[\infty]},\leq,r\rangle$ \cite{To}.

Assume that $Y_t$ mixes $s$ with $t$ and consider the two-coloring $c': [t,Y_t]_{n+1}\to 2$ defined by
\begin{equation*}
c'(p)=
\begin{cases}
1 & \text{ if } \exists q\in [s, Y_s]_{n+1}, Y_t \text{ mixes } p \text { with } q  \text{ and } \phi_t(p)=\phi_s(q),\\
0 & \text{ otherwise.}
\end{cases}
\end{equation*}

The fact that $\langle U, \leq, r\rangle$ is a topological Ramsey space, Corollary $1$, gives us a $Z \leq Y_t$ where $c'\upharpoonright [t,Z]_{n+1}$ is constant. If the constant value is equal to one, then on $Z$ we have that for every $p\in [t,Z]_{n+1}$ there exists a $q\in [s, Z_s]_{n+1}$, where $Z_s=s\cup (Z\setminus t)$, Z mixes $p$ with $q$ and also $\phi_t(p)=\phi_s(q)$. 

If the constant value $c'\upharpoonright [t,Z]_{n+1}$ is equal to zero, for every $p\in [t,Z]_{n+1}$ either there exists $q\in [s,Z]_{n+1}$ so that $Z$ mixes $p$ with $q$ and $\phi_t(p)\neq\phi_s(q)$, or there is not such a $q$. We require $\langle U, \leq, r\rangle$ to satisfy the following property \\ $\boldsymbol{(\mathcal{P})}$: if $c'\upharpoonright [t,Z]_{n+1}=0$, then there exists $Z'\leq Z$ so that $Z'$ separates $s$ with $t$.

In the case that $(\mathcal{P})$ is not satisfied, i.e. for every $Z'\leq Z$, there exist $p\in [t,Z']_{n+1}$ and $q\in [s,Z']_{n+1}$, so that $Z'$ mixes $p$ with $q$ and $\phi_t(p)\neq \phi_s(q)$, no canonization result can be obtained.

Therefore we assume that if $Y$ mixes $s$ with $t$, then $c'\upharpoonright [t,Z]_{n+1}=1$. Observe that if $c'\upharpoonright [t,Z]_{n+1}=1$, then $Z$ mixes $s$ with $t$ by Proposition $2$.

What we have shown is if $\langle U,\leq, r\rangle$ satisfies $(\mathcal{P}$), given any $s,t\in \mathcal{A}U_n$, there exists $Z\leq U$ so that one of the following two possibilities holds. 
\begin{enumerate}
\item{}$Z$ mixes $s$ with $t$ if and only if for every $p\in [t,Z]_{n+1}$ there exists $q\in [s,Z]_{n+1}$ so that $Z$ mixes $p$ with $q$ and $\phi_s(q)=\phi_t(p)$.
\item{}$Z$ separates $s$ with $t$. 
\end{enumerate}

Suppose now that our topological Ramsey space $\langle U, \leq, r\rangle$ has the property that for any $s,t\in \mathcal{AU}_n$ and $X\leq U$ there exists $Y\leq X$ so that $[s,Y]\neq \emptyset$ and $[t,Y]=\emptyset$. As we defined above, if $X$ is so that $[s,X]\neq \emptyset$ and $[t,X]\neq \emptyset$ we say that $X$ is $\it{compatible}$ with $s$ and $t$.
An instance of such a topological Ramsey space forms the space of strong subtrees $\langle \mathcal{S}_{\infty}(U),\leq,r\rangle$, where $U$ in this context is a tree with fixed branching number $b$ but no finite branches \cite{To}. Another is the space of connections $\langle F_{\omega,\omega}, \leq ,r\rangle$ \cite{Vl}.

First we need the following: given $s,t\in \mathcal{A}U_n$, $depth_U(s)=k \leq depth_U(t)=l$ we require our space to have the property that: if $s$ has an end extension $w$, $s\cup w\in [s,U]_{n+1}$, made out of $U[j,m)$, then $t$ has also an end extension $v$, $t\cup v\in [t,U]_{n+1}$, made out of $U[j,m)$, for $l\leq j\leq m$. Suppose not, i.e. there exist two finite approximations $s,t\in \mathcal{A}U_n$, as above, so that for every $X\leq U$, $\exists Y\leq X$ and a finite set $\mathcal{X}\subset [j,m)$ such that $t\cup v\in [t,Y]_{n+1}$ and $v$ is made out of $U_{\mathcal{X}}=\cup_{n\in \mathcal{X}}U(n)$ but there is no $s\cup w\in [s,Y]_{n+1}$ where $w$ is made out of $U_{\mathcal{X}}$. Consider the two-coloring\\ $c':[t,U]_{n+1}\to 2$ defined by 
\begin{equation*}
c'(t\cup v)=
\begin{cases}
1 & \text{ if } \exists s\cup w\in [s,U]_{n+1}, w,v \text{ are both made out of } U_{\mathcal{X}}\\ &  \text{ for some } \mathcal{X}\subset [j,m),\\
0 & \text{ otherwise.}
\end{cases}
\end{equation*}

We get $X\leq U$ so that $c'\upharpoonright [t,X]_{n+1}=0$ and $X$ is compatible with $s$ as well. But then no canonization theorem can be achieved. This is due to the fact that $\phi_s([s,X]_{n+1}\cap  \phi_t([t,X]_{n+1})=\emptyset$ always, even when $X$ mixes $s$ with $t$.
Therefore we can assume that our topological Ramsey space satisfies the above property. We remark here that all the known topological Ramsey spaces satisfy the above assumption. 

We need also to assume that for any $X\leq U$, $s\in \mathcal{A}U_m$, $m\neq 0$, $s\in \hat{\mathcal{F}}$ and $v$ so that $s\cup v\in [s,X]_{m+1}$, there exists $Y\in [s,X]$ so that $s\cup v\notin [s,Y]_{m+1}$. Suppose that there exists $s\in \mathcal{A}U_m$ and $v$ so that $s\cup v\in [s,Y]_{m+1}$, for all $Y\leq U$. This in general is not possible cause our space will violate axiom $A.4$. To see that consider the coloring $c':[s,U]_{m+1}\to 2$ defined by $c(s\cup v)=0$ and $c(p)=1$, for all $p\in [s, U]_{m+1}$, $p\neq s\cup v$. There is no $Y\leq U$ that will satisfy the conclusions of $A.4$. The only exception occurs if $s\cup v$ is the only element of the set $[s,U]_{m+1}$. In that case observe that $\phi_{s}=\emptyset$. 
Therefore we can assume that given $s\in \mathcal{A}U_m$,  $s\in \hat{\mathcal{F}}$ and $v$ so that $s\cup v\in [s,U]_{m+1}$, there exists $Y\in [s, U]$ so that $s\cup v \notin [s,Y]_{m+1}$. In other words we can always go to a reduct that avoids a specific end-extension of $s$ with length $m+1$. If this is not the case, then $\{s\cup v\}=[s,U]_{m+1}$ and $\phi_s=\emptyset$.

Given $s,t\in \mathcal{A}U_n$, recall from above $U[t]=\{ p: t\sqsubseteq p, p=r_j(Y), Y\leq U\}$ and $U[s]=\{ q: s\sqsubseteq q, q=r_i(X), X\leq U\}$. Observe that $[t,U]_{n+1}\subset U[t]$ and $[s,U]_{n+1}\subseteq U[s]$ respectively. We have also assumed that $depth_U(s)=k<depth_U(t)=l$. Let $A=\{q \in [s,U]_{n+1}: q\text{ made of } U[l,m)\}$, where $m$ is the smallest $depth$ of all the length one end extensions of $t$.

At this stage we must assume that our space has the property that there exists an one-to-one map $\iota: \cup [t,U]_{n+1}\to \cup [s,U[q]]_{n+1}$, $q\in A$, satisfying the following properties:
\begin{enumerate}
\item{} If $t\sqsubseteq p$, then $\iota(t)\sqsubseteq \iota(p)$,
\item{} If $depth_U(p)=m$, then $depth_U(\iota(p))=m$,
\item{} If $p\in [t,U]_{n+1}$ and $p(n)$ is made out of $ U[l,m)$, then $\iota(p)\in [s,U[q]]_{n+1}$ and $\iota(p)(n)$ is made out of $U[l,m)$.
\end{enumerate}

Assuming the existence of such a map and that $U$ mixes $s$ with $t$, consider the two-coloring $c':[s,U[q]]_{n+1}\to 2$ defined by
\begin{equation*}
c'(q')=
\begin{cases}
1 & \text{ if } \exists p \in [t,U[t]]_{n+1} \text{ such that } \iota (p)=q', \\
    &U \text{ mixes } p \text{ with } q' \text{ and } \phi_s(q')=\iota( \phi_t(p)),\\
0 & \text{ otherwise.}
\end{cases}
\end{equation*}

Once again, we get a $Z\leq U$, $Z[q]\neq \emptyset$, $Z[t]\neq \emptyset$ so that either $$c'\upharpoonright [s, Z[q]]_{n+1}=1\text{ or }c'\upharpoonright[s, Z[q]]_{n+1}=0.$$ In the first case $Z$ mixes $s$ with $t$ and for every $q'\in [s,Z[q]]_{n+1}$, there exists $ p\in [t, Z]_{n+1}$ such that $Z$ mixes $p$ with $q'$ and $\phi_s(q')=\iota( \phi_t(p))$. 
We require here, as above, that $\langle U, \leq, r\rangle$ satisfies property $(\bf{\mathcal{P}})$, so the alternative $c'\upharpoonright [s, Z[q]]_{n+1}=0$ is excluded.

Up to this point we have shown that the topological Ramsey space $\langle U, \leq, r\rangle$ has the property that given $s,t \in \hat{ \mathcal{F}}$, there exists $Z\leq U$ that mixes $s$ with $t$ if and only if $\phi_s$ agree with $\phi_t$, up to $\iota$, otherwise $Z$ separates $s$ with $t$. By Proposition $1$ there exists $X\leq Z$ with the property that given any $s,t\in \mathcal{A}X_n$ one has that either $X$ mixes $s$ with $t$ if and only if for every $p\in [t,X]_{n+1}$ there exists $q\in [s,X]_{n+1}$ so that $X$ mixes $p$ with $q$ and $\phi_s(q)=\phi_t(p)$, up to $\iota$, or $X$ separates $s$
 with $t$. 
 From now on we will denote $\mathcal{F}$ instead of $\mathcal{F}\upharpoonright X$ and $ \hat{ \mathcal{F}}\setminus \mathcal{F}$ instead of $ \hat{ \mathcal{F}}\setminus \mathcal{F}\upharpoonright X$ since everything is taking place below $X$. We will also omit the $\iota$.

 We are now ready to define the inner map $\phi$ that is going to witness the coloring being canonical. For $s\in \hat{\mathcal{F}}$, $|s|=n$, we define $$\phi(s)=\bigcup_{s'\sqsubseteq s}\phi_{s'}(s(|s'|))=(\phi_{r_0(t)}(r_1(t)), \phi_{r_1(t)}(t(2)), \dots, \phi_{r_{n-1}(t)}(t(n))).$$
  
  Next we have to show the following four lemmas.
  
  \begin{lemma}The following are true for all $Y\leq X$.
  \begin{enumerate}
 \item{} Let $s,t \in \hat{ \mathcal{F}}\setminus \mathcal{F}$. If $\phi_s\neq \emptyset$ and $\phi_t=\emptyset$, there exists $w\in [s, X]_{|s|+1}$ so that $X$ mixes $t$ with $s\cup w$ with at most one equivalence class of $ [s, X]_{|s|+1}$.
  
  \item{} If $X$ separates $s$ with $t$, then its separates $s\cup w$ with $t\cup v$ for all $w\in  [s, X]_{|s|+1}$ and $v\in  [t, X]_{|t|+1}$.
  
  \item{} If $s\sqsubset t$, $s,t\in \mathcal{F}$  and $\phi(s)=\phi(t)$, then $X$ mixes $s$ with $t$.
  \end{enumerate}
    \end{lemma}
    
    \begin{proof}
    Suppose that $X$ mixes $t$ with $s\cup w$, and also $t$ with $s\cup v$, and $\phi_s(s \cup w)\neq \phi_s(s\cup v)$. By Lemma $1$, we get that $X$ mixes $s\cup w$ with $s\cup v$, a contradiction.
    
    Suppose that $X$ separates $s$ with $t$ and there exists $s\cup w\in [s, X]_{|s|+1}$ and $t\cup v\in  [t, X]_{|t|+1}$ so that $X$ mixes $s\cup w$ with $t\cup v$. This means that for every $Y\leq X$ there exists $w'\in \mathcal{F}_{s\cup w}\upharpoonright Y$ and $v'\in \mathcal{F}_{t\cup v}\upharpoonright Y$ so that $Y$ mixes $s\cup w\cup w'$ with $t\cup v \cup v'$, a contradiction to our assumption that $X$ separates $s$ with $w$. 
    
    Suppose now that  $s,t\in \mathcal{F}$, $s\sqsubset t$ and $\phi(s)=\phi(t)$. This means that for all $j\in [|s|,|t|]$, $\phi_{r_j(t)}=\emptyset$, which, by an induction on $n= [|s|,|t|]$, implies that $s$ gets mixed by $X$ with all the extensions of $r_{|s|}(t)$. In particular $X$ mixes $s$ with $t$.
    
    \end{proof}
  
  \begin{lemma}
  For $s,t\in \hat{\mathcal{F}}$ if $\phi(s)=\phi(t)$, then $X$ mixes $s$ with $t$. In particular if $s,t\in \mathcal{F}$ and $\phi (s)=\phi (t)$, then $c(s)=c(t)$.
  \end{lemma}
  
  \begin{proof} The proof is by induction on $l< \max (depth_X(s), depth_X(t))$. For $l=0$, $s\cap r_0(U)=t\cap r_0(X)=\emptyset$, so $X$ mixes $s\cap r_0(X)$ with $t\cap r_0(X)$. Assume that $X$ mixes $s\cap r_{l-1}(X)$ with $t\cap r_{l-1}(X)$ and consider $s\cap r_l(X)$ and $t\cap r_l(X)$. If $s\cap r_l(X)\neq \emptyset$ and $t\cap r_l(X)=\emptyset$, then we must have that $\phi_{s\cap r_l(X)}=\emptyset$. This implies that $s\cap r_l(X)$ is mixed with $s\cap r_{l-1}(X)$ which is mixed with $t\cap r_l(X)$. Therefore $s\cap r_l(X)$ is mixed with $t\cap r_l(X)$. Similarly if $s\cap r_l(X)= \emptyset$ and $t\cap r_l(X)\neq \emptyset$. 
  \end{proof}

\begin{lemma}
For $s,t\in \hat{\mathcal{F}}$, $s\neq t$ it doesn't hold that $\phi(s)\sqsubseteq \phi(t)$.
  \end{lemma}

\begin{proof}Suppose that there are $s,t\in \hat{\mathcal{F}}$ with $\phi(s)\sqsubset \phi(t)$. Let $j<\omega$ be so that $\phi(s)=\phi(r_j(t))$. There is at least one $i\in \omega, i>j$ so that $\phi_{r_i(t)}\neq \emptyset$. Assume that $\phi_{r_{j}(t)}\neq \emptyset$, which implies that $X$ mixes $s$ with $r_{j}(t)\cup v$, for some $v$ that belongs to the equivalence relation on $[r_j(t),X]_{j+1}$ induced by $\phi_{r_j(t)}$. But then consider a reduct $Y\leq X$ that avoids $v$. Then $Y$ separates $s$ with $r_j(t)$, a contradiction. If no such a reduct $Y$ is possible to be fund, it implies that $r_j(t) \cup v=[r_j(t),X]_{j+1}$, which implies that $\phi_{r_j(t)}=\emptyset$.
\end{proof}
                   
    \begin{lemma}
    For $s,t\in \mathcal{F}$, if  $c(s)=c(t)$, then $\phi(s)=\phi(t)$.
  \end{lemma}                   
                       
 \begin{proof}
 Let $s,t\in \mathcal{F}$ with $c(s)=c(t)$. Then for every $l<\max (depth_X(s), depth_X(t))$, $X$ mixes $s\cap X(l)$ with $t\cap X(l)$. We show by induction that for all such an $l$ it holds that $\phi(s\cap X(l))=\phi(t\cap X(l))$. For $l=0$ $s\cap X(0)=t\cap X(0)=\emptyset$. Assume that $\phi(s\cap X(l-1))=\phi(t\cap X(l-1))$ and consider $s\cap X(l)$ and $t\cap X(l)$. Assume that $s\cap X(l)\neq \emptyset$ and $t\cap X(l)=\emptyset$. If $\phi_{s\cap X(l-1)}\neq \emptyset$ then we must have that $t\neq t\cap X(l)$. If that was the case it will contradict the above lemma since $\phi(t)=\phi(t\cap X(l))\sqsubset \phi(s)$. Notice that $X$ mixes $s\cap X(l-1)$ with $t\cap X(l-1)$, since $\phi(s\cap X(l-1))=\phi(t\cap X(l-1))$. This implies that $\phi_{s\cap X(l-1)}=\phi_{t\cap X(l-1)}$. But $s\cap X(l)\neq \emptyset$, $\phi_{s\cap X(l-1)}\neq \emptyset$ and $t\cap X(l)=\emptyset$, a contradiction.
 \end{proof}    
 
 Obviously $\phi$ is an Inner mapping. Next we prove that $\phi$ is maximal among all other mappings representing $f:\mathcal{F}\upharpoonright X\to \omega$.
 
 \begin{proposition} Suppose $Y\leq X$ and there is another $\phi'$ map, satisfying the condition that for all $t_0,t_1\in \mathcal{F}\upharpoonright Y$ $f(t_0)=f(t_1)$ if and only if $\phi'(t_0)=\phi'(t_1)$. Then there exists $Z\leq Y$ so that for every $s\in \mathcal{F}\upharpoonright Z$ $\phi'(s) \subseteq \phi(s)$.
 \end{proposition}        
 
 \begin{proof}By Corollary $2$ and Proposition $1$, we can assume that $\phi'$ has the form of Definition $2$. To see this, for any $t\in \mathcal{F}\upharpoonright X$, $i<|t|$, by Corollary $2$ there exists $X'\in [r_{i}(t),X]$ so that for every $s\supset r_{i}(t)$, $\phi'(s)\cap s(i)=g(s(i))$, for $g$ an inner map for $r_i(s)$, as in $A.4^\star$. If this is done for an arbitrary $t\in \mathcal{F}\upharpoonright X$, by Proposition $1$ we can assume that it holds for every $t\in \mathcal{F}\upharpoonright X$.
 Pick $t\in \mathcal{F}\upharpoonright Y$.

  Pick $t\in \mathcal{F}\upharpoonright Y$. Let $n=|t|$. For $i<n$ consider both $\phi'_{r_i(t)}$ and $\phi_{r_i(t)}$. Consider the two coloring $c':[r_i(t),Y]_{i+1}\to 2$, defined by
 \begin{equation*}
c'(v)=
\begin{cases}
1 & \text{ if } \phi'_{r_i(t)}(v)\subseteq \phi_{r_i(t)}(v),\\
0 & \text{ otherwise.}
\end{cases}
\end{equation*}

There exists $Z'\leq [r_i(t),Y]$ so that $c'\upharpoonright [r_i(t),Z']_{i+1}$ is constant and equal to one. Observe that we can only have for every extension $v$ of $r_i(t)$, so that $r_i(t)\cup v\in [r_i(t),Z']_{i+1}$, $\phi'_{r_i(t)}(v)\subseteq \phi_{r_i(t)}(v)$. This is due to the fact that both $\phi'$ and $\phi$ witness the same $f\upharpoonright (\mathcal{F}\upharpoonright Y)$. By Proposition $1$, there exists $Z\leq Z'$ that satisfies the conclusions of our proposition.
 
 \end{proof}  
 Therefore $\phi$ satisfies the conditions of Definition $2$. Then the proof of Theorem $1$ is complete in the case of transitive mixing.
 
 \subsection{non-transitive mixing}
 
 Now we deal with the case that mixing is not transitive. Lemma $1$ shows that in the case of equal depth, mixing is guaranteed to be transitive. As its proof points out, in this case the mixing is forced to take place on the "tail" of $X$. The counter examples indicates that there are topological Ramsey spaces where the mixing is not necessarily taking place on the "tail" of $X$. 
 This necessitates the following definition.
  
 \begin{definition} Let $f:\mathcal{F}\to \omega$, where $\mathcal{F}$ is a front on $X$. For $s,t\in \hat{\mathcal{F}}$ we say that $X$ weakly mixes $s$ with $t$, if $depth_X(s)<depth_X(t)$ and there exist $ w_s^t\in \mathcal{L}X$, $w_s^t\in t\setminus s$, so that $t'\cup w_s^t  \sqsubseteq t$, such that for every $Y\leq X$, $[s,Y]\neq \emptyset$, $[t,Y]\neq \emptyset$, there exists $\bar{s},\bar{t}\in \mathcal{F}\upharpoonright Y$ where $s \sqsubseteq \bar{s}$, $\bar{s}(|s|)\in \langle w_s^t,v\rangle_s$, $v\in \mathcal{L}X$, $t\sqsubseteq \bar{t}$ and $f(\bar{s})=f(\bar{t})$. 

\end{definition}

 %We remind the reader that for $t\in \mathcal{A}Y$, by $\mathcal{L}Y/t$ we denote the set $\{ v\in \mathcal{L}Y: v \subset Y/t \}$.
Notice that if $X$ weakly mixes $s$ with $t$, then $X$ mixes $s$ with $t$. Axiom $A.2$ guarantees that the cardinality of the set $ \langle w_s^t,v\rangle_s$ is finite for any $v\in \mathcal{L}X$.

 A topological Ramsey space is said to have {\it{full initial segments}} if any $s\in \mathcal{A}U_m$, $m\in \omega$, $s\leq_{fin}r_l(U)$, for some $l\geq m$, then $s$ is actually made out of $\cup_{n\in [1,l)}U(n)$. In other words we have that any $s\in \mathcal{A}U$, with $depth_U(s)=l$ is made out of $\cup_{n\in [1,l)}U(n)$. Not a level of $U$ below $l$ can be missed.
 Next we claim the following.

\begin{claim}Weak mixing is taking place on topological Ramsey spaces where the end extensions can be made out of more than one level and do not have full initial segments.
\end{claim}
\begin{proof}Suppose that $\langle U, \leq, r\rangle$ has the property that end extensions are made out of only one level.  Let $\mathcal{F}=\mathcal{A}X_2$, $X\leq U$, $s\in \mathcal{A}X_1$ and $\phi_s\neq \emptyset$. Assume also that  for a fixed $w'\in \mathcal{A}(X/s)_1$, $w'\in \langle w'' \rangle_{\emptyset}$, $\phi_{ w'}=\emptyset$ and $f(s\cup w'')=f(w'\cup v)$. As a consequence $f(s\cup w'')=f(w' \cup v)$, for every $v $ so that $w'\cup v\in \mathcal{F}\upharpoonright X$. 
Observe that any $Y\leq X$, so that $s\cup w''\in \mathcal{A}Y$, $Y$ mixes $s$ with $w'$, but $\phi_s\neq \phi_{w'}$. In fact in this case the mixing of $s$ and $ w'$ is already decided from the $depth_X(s\cup w'')$, and is irrelevant to the "tail" of $X$. According to Definition $4$ $X$ weakly mixes $s$ with $w'$ and $w_s^{w'}=w'$, $v=w''\setminus w' $, $v\notin \mathcal{L}X/w'$. Observe that the above happens cause $\phi_{w'}=\emptyset$. Let $c:\mathcal{A}X_1\to 2$ defined by 
\begin{equation*}
c(w)=
\begin{cases}
1 & \text{ if } \phi_w=\emptyset ,\\
0 & \text{ otherwise.}
\end{cases}
\end{equation*}

There exists $Y\leq X$ so that $c\upharpoonright \mathcal{A}Y_1$ is constant. On $Y$ instances of $s$ and $w'$ as above do not occur. In other words we have that $v\in \mathcal{L}X/w'$ in Definition $4$.
Let $F$ be a front of the form $\mathcal{A}X_n$, for $n\in \omega$. Then by induction on $n$ and a coloring as above, we can also assume that there are no $s$, $w'$ getting weakly mixed by $X$. Therefore we can assume that for any front of the form $\mathcal{A}X_n$, $n\in \omega$, weak mixing does not occur.
 
 Next suppose that $\langle U, \leq, r\rangle$ has full initial segments. If $X$ weakly mixes $s$ with $t$, $X\leq U$, then $X$ actually mixes $s$ with $t$. To see this, let $w_s^t$ witness that $X$ weakly mixes $s$ with $t$. Then any end extension $s'$ of $s$ is so that $s'(|s|)\in \langle w_s^t,v \rangle_s$. Therefore $X$ mixes $s$ with $t$. 

\end{proof}
 
 We observe the following. 

\begin{claim} If $s$ is weakly mixed by $X$ with $t$ and $t$ is also weakly mixed with $p$, then $X$ weakly mixes $s$ with $p$ as well and $w_s^t\subseteq w_s^p$. Similarly in the case that $X$ weakly mixes $s$ with $t$, mixes, not weakly, $t$ with $p$, then $X$ weakly mixes $s$ with $p$ and $w_s^t\subseteq w_s^p$. 

\end{claim}
\begin{proof}In the first case, by definition, we have that $depth_X(s)< depth_X(t) < depth_X(p)$. Suppose $X$ weakly mixes $s$ with $t$, then for $w^t_s$ we have that for every $Y\leq X$, compatible with both $s,t$, there exists $\bar{s},\bar{t}\in \mathcal{F}\upharpoonright Y$ where $s \sqsubseteq \bar{s}$, $\bar{s}(|s|)\in \langle w_s^t,v\rangle_s$ and $t\sqsubseteq \bar{t}$ and $f(\bar{s})=f(\bar{t})$. Consider the coloring $c:[t,X]_{|t|+1}\to 2$ defined by 
\begin{equation*}
c(t\cup v)=
\begin{cases}
1 & \text{ if } X \text{ mixes } t\cup v \text{ with } s'\sqsupseteq s , |s'|=|s|+1, s'(|s|) \in \langle w_s^t,v' \rangle_s, \\
0 & \text{ otherwise.}
\end{cases}
\end{equation*}

There is $X_0\leq X$ so that $c\upharpoonright [t,X_0]_{|t|+1}=1$. By a similar coloring we get a further reduct $X_1\leq X_0$, $w^p_t$, so that every $p\cup v\in [p,X_1]_{|p|+1}$ is weakly mixed by $X_1$ with $t'\sqsupseteq t$ so that $|t'|=|t|+1$, 
$ t'(|t|) \in \langle w_t^p,v'' \rangle_t$. Then $X_1$ mixes $s$ with $p$. If $X_1$ mixes, but doesn't mixes weakly, $p$ with $s$, then we would have that for every $v\in  X_1/p $ there exists $v'\in  X_1/p $ so that $X_1$ mixes $p\cup v$ with $s\cup v'$ and $p\cup v$ with $t'\sqsupseteq t$, $|t'|=|t|+1$. As a result $X$ mixes $s\cup v'$ with $t'$, a contradiction. Therefore there exists $w^p_s$ so that for every $v\in  X_1/p  $, there exists $s', t'$, where $X_1$ mixes $p\cup v$ with $s'\sqsupseteq s$, $|s'|=|s|+1$, $s'(|s|) \in \langle w_s^p,v'\rangle_s$ and $p\cup v$ with $t'\sqsupseteq t$, $|t'|=|t|+1$, $t'(|t|)\in \langle w_t^p, v''\rangle_t $. This implies that $X$ mixes $s'$ with $t'$ which implies that $w^t_s\subseteq w_s^p$.

Now in the case that $X$ mixes, but not weakly, $t$ with $p$, then $X$ mixes $s$ with $p$. If $s$ and $p$ are mixed by $X$, not weakly, then we can assume that for every $v\in  X_1/p  $, there exists $s\cup v'\in [s,X]_{|s|+1}, t\cup v''\in [t,X]_{|t|+1}$ so that $X$ mixes $p\cup v$ with both $t\cup v''$ and $s\cup v'$, contradicting that $s$, $t$ are weakly mixed. Therefore there exists $w^p_s$ so that for every $v\in  X_1/p $, there exists $t\cup v''\in [t,X]_{|t|+1}$, so that $X_1$ mixes $p\cup v$ with $s'\sqsupseteq s$, $s'(|s|) \in \langle w_s^p,v\rangle_s$, and $p\cup v$ with $t\cup v''$. As a result $X_1$ mixes $s'$ with $t\cup v''$ as well. This implies that $w_s^t\subseteq w_s^p$.
\end{proof}

 The above claim shows the following. Let $s,t\in \hat{ \mathcal{F}}\setminus \mathcal{F}$ so that are weakly mixed by $X$, i.e. there is $w_s^t$ and extensions $s \sqsubset \bar{s}$ and $t\sqsubseteq \bar{t}$ satisfying Definition $4$, $\bar{s},\bar{t}\in \mathcal{F}$, so that $f(\bar{s})=f(\bar{t})$. Any $p\in \hat{ \mathcal{F}}\setminus \mathcal{F}$ where there exists $p\sqsubseteq \bar{p}$ with $f(\bar{s})=f(\bar{t})=f(\bar{p})$, is so that $w_s^t\subseteq p\setminus s$ as well. Let $s=r_n(X)$, for some $X\leq U$, observe that $w_s^t$ is in the second part of $s\cup v=\bar{s}$, but on $t,p$ wich form the first part of $\bar{p}$ and $\bar{t}$. As a consequence $f$ factors through $\langle \rangle_{}$ non trivially. Therefore the only way to ruin transitivity is through the $\langle \rangle_{}$ operation. In this case the map $\phi$ in the Definition $2$ is so that for at least one of its components the $\langle \rangle$ is non trivial.

 \end{proof}
 
 Consider any topological Ramsey space $\langle U,\leq,r \rangle$ that satisfied the strengthened version $A.4^\star$, property $\mathcal{P}$ and for any $s,t\in \mathcal{A}U_n$ it holds that, for any $s\cup w\in [s,U]_{n+1}$, where $w$ is made out of $\cup_{n\in \mathcal{X}}U(n)$, $\mathcal{X}\subseteq [k,l)$, there exists $t\cup v\in [t,U]_{n+1}$ with $v$ made out of exactly the same levels $\cup_{n\in \mathcal{X}}U(n)$. Then given any front $\mathcal{F}$ on $X\leq U$ and $f:\mathcal{F}\to \omega$, there exists $Y\leq X$ and an Inner map $\phi$ on $\mathcal{F}\upharpoonright Y$, so that for every $s,t \in \mathcal{F}\upharpoonright Y$ it holds that $f(s)=f(t)$ if and only if $\phi(s)=\phi(t)$. In the case of non-transitive mixing, the above assertion holds for any front of the form $\mathcal{A}Y_n$, $n\in \omega$.


\begin{thebibliography}{5}
\bibitem[Do-To]{Do-To}  A Ramsey-classification theorem and its application in the Tukey theory of ultrafilters, Transactions of AMS volume 366, number 3 1659-1684.
\end{thebibliography}
     \end{document}